\documentclass[12pt, leqno,twoside]{article}
\usepackage{amssymb, url, color, graphicx, amscd, mathrsfs}

\usepackage[colorlinks=true, bookmarks=true, pdfstartview=FitH, linktocpage=true, linkcolor = blue, citecolor = blue]{hyperref}
\usepackage{amssymb}
\usepackage{amsmath}
\usepackage{accents}
\usepackage{amsthm}
\usepackage{graphicx}
\usepackage{cite}

\usepackage{amsmath}
\usepackage{amsfonts}
\usepackage{amssymb}
\usepackage{amsmath,bbm,amssymb,amsxtra}
\usepackage{mathrsfs}
\usepackage{enumerate}
\usepackage{caption}
\allowdisplaybreaks
\usepackage{epsfig}
\usepackage{ae}

\def\Xint#1{\mathchoice
	{\XXint\displaystyle\textstyle{#1}}%
	{\XXint\textstyle\scriptstyle{#1}}%
	{\XXint\scriptstyle\scriptscriptstyle{#1}}%
	{\XXint\scriptscriptstyle\scriptscriptstyle{#1}}%
	\!\int}

\def\XXint#1#2#3{{\setbox0=\hbox{$#1{#2#3}{\int}$ }
		\vcenter{\hbox{$#2#3$ }}\kern-.6\wd0}}
\def\dashint{\Xint-}

\newcommand{\essinf}{\text{ess inf}}
\newcommand{\esssup}{\text{ess sup}}

\newcommand{\ox}{{\bar{x}}}
\newcommand{\ux}{{\underline{x}}}

\newcommand{\R}{\mathbb R}

\newcommand{\rarrow}{\rightarrow}
\newcommand{\Ap}{$\text{A}_\text{p}$}
\newcommand{\Aone}{$\text{A}_\text{1}$}
\newtheorem{thm}{Theorem}[section]
\newtheorem{lem}[thm]{Lemma}
\newtheorem{prop}[thm]{Proposition}
\newtheorem{rem}[thm]{Remark}

\newtheorem{defn}[thm]{Definition}
\newtheorem{example}[thm]{Example}
\numberwithin{equation}{section}
\usepackage{comment}
\begin{document}
\title{\Large\bf  Admissibility versus \Ap-conditions on regular trees
\footnotetext{\hspace{-0.35cm}
$2010$ {\it Mathematics Subject classfication}: 30L99, 31C45, 46E35
\endgraf{{\it Key words and phases}: $\text{A}_\text{p}$-condition, doubling measure, Poincar\'e inequality, regular tree}
 \endgraf{Authors have been supported by the Academy of Finland via  Centre of Excellence in Analysis and Dynamics Research (project No. 307333).}
 }
}
\author{Khanh Ngoc Nguyen,  Zhuang Wang}
\date{ }
\maketitle
\begin{center}
\emph{Dedicated to Professor Pekka Koskela on the occasion of his 59th birthday celebration}
\end{center}
\begin{abstract}
We show that the combination of doubling and $(1,p)$-Poincar\'e inequality is equivalent to a version of the \Ap-condition on rooted $K$-ary trees.
\end{abstract}
\section{Introduction}
The class of $p$-admissible weights for Sobolev spaces and differential equations on $\R^n$ was introduced in \cite{HKM}. The definition was initially based on four conditions, but Theorem 2 in \cite{PP1995} and Theorem 5.2 in \cite{HK95} reduce them to the following two conditions, see also \cite[2nd ed., Section 20]{HKM}. 
\begin{defn}\rm 
A measure $\mu$ on $\R^n$ is {\it $p$-admissible}, $1\leq p<\infty$, if it is doubling and supports a $(1, p)$-Poincar\'e inequality. If $d\mu=w \,dx$, we also say that the weight $w$ is $p$-admissible.
\end{defn}
Here $\mu$ supports a $(q, p)$-Poincar\'e inequality, $1\leq q<\infty, 1\leq  p<\infty$, if there is a constant $C>0$ such that 
\[ \left(\dashint_{B(x, r)}|u-u_{B(x, r)}|^q\, d\mu\right)^{1/q}\leq C r\left(\dashint_{B(x, r)} |\nabla u|^p\, d\mu\right)^{1/p}
\] 
for every $u\in C^1(\mathbb R^n)$, every $x\in \R^n$ and all $r>0$.

In \cite[Section 15]{HKM}, it was shown that Muckenhoupt $\text{A}_\text{p}$-weights are $p$-admissible, but the converse is not true in $\R^n$, $n\geq 2$, see also \cite{Jana01}.  Surprisingly, on the real line $\R$, any $p$-admissible measure is actually given by an $\text{A}_\text{p}$-weight, see \cite{BBK06}. Very recently, it was also shown in \cite{BBS19} that a measure on $\R$ is locally $p$-admissible if and only if it is given by a local $\text{A}_\text{p}$-weight. Moreover, on $\R^n, p$-admissible measures can be characterized by a stronger version of the Poincar\'e inequality, the $(q, p)$-Poincar\'e inequality with $q>p$. Under doubling, the $(1, p)$-Poincar\'e inequality improves to a $(q, p)$-Poincar\'e inequality with $q>p$ by \cite{PP1995} and any measure satisfying $(q, p)$-Poincar\'e inequality with $q>p$ is a doubling measure, see \cite{AH19} and \cite{KMR15}.

In the recent years, analysis on regular trees has been under development, see \cite{BBGS,KW19,PKZ,khanh,khanhzhuang}. Given  a $K$-regular tree $X$ (a rooted $K$-ary tree), $K\geq 1$,  we introduce a metric structure on $X$ by considering each edge of $X$ to be an isometric copy of the unit interval. Then the distance between two vertices is the number of edges needed to connect them and there is a unique geodesic that minimizes this number.  Let us denote the root by 0. If $x$ is a vertex, we define $|x|$ to be the distance between 0 and $x.$ Since each edge is an isometric copy of the unit interval, we may extend this distance naturally to any $x$ belonging to an edge. 

Write $d|x|$ for the length element on $X$ and let $\mu:[0,\infty)\to (0,\infty)$ be a locally integrable function. We abuse notation and refer also to the measure generated via $d\mu(x)=\mu(|x|)d|x|$ by $\mu.$ Further, let $\lambda:[0,\infty)\to (0,\infty)$ be locally integrable and define a distance via $ds(x)=\lambda(|x|)d|x|$ by setting $d(z,y)=\int_{[z,y]}ds(x)$
whenever $z,y\in X$ and $[z,y]$ is the unique geodesic between $z$ and $y$. We abuse the notation and let $\mu(x)$ and $\lambda(x)$ denote $\mu(|x|)$ and $\lambda(|x|)$, respectively, for any  $x\in X$, if there is no danger of confusion. Throughout this paper,  we assume additionally that the diameter of $X$ is infinity. 

Our space $(X, d, \mu)$ is a metric measure space and hence one may define a Newtonian Sobolev space $N^{1, p}(X):=N^{1, p}(X, d, \mu)$ based on upper gradients \cite{HK} and \cite{N00}.
 It is then natural to ask if we can characterize the $p$-admissibility of a given $\mu$, see Section \ref{adm} for the definitions.  To do so, we introduce the following $\text{A}_\text{p}$-conditions on regular trees.

Before continuing, we first introduce some notations. For any $x\in X$ and $r>0$, we denote by $\ox^r$ the point in $[0, x]$ with $d(\ox^r, x)=\min\{r, d(0,x)\}$ 
and denote by $\ux_r$ a point in $X$ such that $x\in [0, \ux_r]$ with $d(\ux_r, x)=r$. Hence $\ox^r$ is an ancestor of $x$ and  $\ux_r$ is a descendant of $x$, see Section \ref{section2.1} for more relations between points on regular trees. Also let
\[F(x,r)=\{y\in X: x\in[0,y], d(x,y)<r\}
\]be the downward directed ``half ball".
It is perhaps worth to mention that the notations $\ox^r$ and $F(x,r)$ coincide with the notation ``$z$" and $F(x, r)$ in \cite[Lemma 3.2]{BBGS}, respectively. 

Given $1<p<\infty$, we set
\begin{equation}
\label{ap}\text{A}_\text{p}(x, r)=\frac{\mu(F(\ox^r, 2r))}{2r}\cdot\left (\frac{1}{r}\int_{[x, \ux_r]}\left (\frac{K^{j(w)-j(x)}\mu(w)}{\lambda(w)}\right )^{\frac{1}{1-p}}\,ds(w)\right )^{p-1}
\end{equation}
and we define
\begin{equation}\label{a1}
\text{A}_\text{1}(x, r) =\frac{\mu(F(\ox^r, 2r))}{2r}\cdot \esssup_{w\in [x, \ux_r]}\frac{\lambda(w)}{K^{j(w)-j(x)}\mu(w)}
\end{equation}
where $j(w)$ and  $j(x)$ are the smallest integers such that $j(w)\geq |w|$ and $j(x)\geq |x|$ , respectively.
Notice that $\text{A}_\text{p}(x, r)$ is independent of the choice of $\ux_r$ among the points $y$ with $x\in [0, y]$ and $d(y, x)=r$.
\begin{defn} \rm\label{def-ap}Let $1\leq p<\infty$ and $X$ be a $K$-regular tree with distance $d$ and metric $\mu$. We say that $\mu$ satisfies the {\it \Ap-condition} if 
\begin{equation}\label{Ap1}
\sup\left\{\text{A}_\text{p}(x, r): x\in X, r>0\right\}<\infty.
\end{equation}
We say that $\mu$ satisfies the {\it $\text{A}_\text{p}$-condition far from $0$} if 
\begin{equation}\label{Ap2}
\sup\left\{\text{A}_\text{p}(x, r): x\in X, 0<r\leq 8\, d(0, x)\right\}<\infty.
\end{equation}
\end{defn}

If $K=1$ and $\lambda\equiv 1$, then the $1$-regular tree $(X, d, \mu)$ is isometric to the half line $(\R^+, dx, \mu\, dx)$ and our $\text{A}_\text{p}$-condition \eqref{Ap1} is equivalent to $\mu$ being a Muckenhoupt $\text{A}_\text{p}$-weight, see \cite{HKM,Jana01,BBK06,BBS19} for more information about Muckenhoupt $\text{A}_\text{p}$-weights.
Above, we call \eqref{Ap2} ``$\text{A}_\text{p}$-condition far from $0$" since $0<r\leq 8\, d(0, x)$ is equivalent to $d(0, x)\geq r/8>0$, which means that $x$ has to be ``far" away from the root $0$ in terms of $r$. 

The main result of this paper is the following characterization of $p$-admissibility on regular trees.
\begin{thm}\label{main theorem} Let $1\leq p<\infty$ and $X$ be a $K$-regular tree with distance $d$ and measure $\mu$. Then we have:
	\begin{enumerate}
		\item For $K=1$,  $\mu$ is $p$-admissible if and only if $\mu$ satisfies the \Ap-condition far from $0$.
		\item 
		For $K\geq 2$, $\mu$ is $p$-admissible if and only if $\mu$ satisfies the \Ap-condition.
	\end{enumerate}
\end{thm}

The characterizations for $K=1$ and $K\geq 2$ are different. For $K\geq 2$, a $K$-regular tree has a kind of symmetry property with respect to the root $0$, since the root has more than one branch. But for $K=1$, the root $0$ behaves like an end point. 

Readers who are familiar with the results on the real line $\R$ may regard our $K$-regular tree with $K\geq 2$ as a generalized model of the real line $\R$. As a byproduct, a slightly modified proof of Theorem \ref{main theorem} for $K\geq 2$ gives a new proof of \cite[Theorem 2]{BBK06}. On the other hand, for $K=1$, one
may connect the result on $1$-regular trees with the result on bounded intervals (see \cite[Theorem 4.6]{BBS19} for bounded intervals). Hence Theorem \ref{main theorem} is new and interesting even when $K=1$ and $\lambda\equiv 1$, since it gives a full characterization of $p$-admissibility on the half line $\R^+$. 

 In \cite[Example 4.7]{BBS19}, one can find a weight $\omega$ on the interval $[0, 1]$ which is $1$-admissible but not a  Muckenhoupt $\text{A}_\text{1}$-weight on $(0, 1)$. By a suitable constant extension of $\omega$ on $(1, \infty)$, we obtain a weight $\omega'$ which is  $1$-admissible but not a  Muckenhoupt $\text{A}_\text{1}$-weight on $\R^+$.  As evidence towards Theorem \ref{main theorem} for $K=1$, it is easy to check that the extended weight $\omega'$ on $\R^+$ satisfies the $\text{A}_\text{1}$-condition far from $0$, i.e., condition \eqref{Ap2} holds. We refer to \cite{BBS19} and \cite{CW00} for more details.

Let us close this introduction by pointing out that the constant ``$8$" in \Ap-condition far from $0$ \eqref{Ap2} is not necessary. Actually replacing $8$ by any  constant $\infty>c>1$, Theorem \ref{main theorem} for $K=1$ holds. Here the requirement of $c>1$ is sharp in the sense that there exists an example $(\mathbb R_+, dx, \mu \, dx)$ such that \eqref{Ap2} holds for any positive constant $c'<1$ replacing $8$, but $\mu$ is not even doubling, see Remark \ref{remark} and Example \ref{example}. 



The paper is organized as follows. In section \ref{section2}, we introduce regular trees, $p$-admissibility and Newtonian spaces on our tree. We give the proof of Theorem \ref{main theorem} for $K\geq 2$ in Section \ref{proof_2} and the proof of Theorem \ref{main theorem} for $K=1$ is given in Section \ref{section4}. 

\section{Preliminaries}\label{section2}
Throughout this paper, the letter $C$ (sometimes with a subscript) will denote positive constants; if $C$ depends on $a,b,\ldots$, we write
$C=C(a,b,\ldots)$.

\subsection{Regular trees and their boundaries}\label{section2.1}
 {A {\it graph} $G$ is a pair $(V, E)$, where $V$ is a set of vertices and $E$ is a set of edges.   We call a pair of vertices $x, y\in V$  neighbors if $x$ is connected to $y$ by an edge. The degree of a vertex is the number of its neighbors. The graph structure gives rise to a natural connectivity structure. A {\it tree} is a connected graph without cycles. A graph (or tree) is made into a metric graph by considering each edge as a geodesic of length one.

We call a tree $X$ a {\it rooted tree} if it has a distinguished vertex called the {\it root}, which we will denote by $0$. The neighbors of a vertex $x\in X$ are of two types: the neighbors that are closer to the root are called {\it parents} of $x$ and all other  neighbors  are called {\it children} of $x$. Each vertex has a unique parent, except for the root itself that has none. 

A {\it $K$-ary tree}  is a rooted tree such that each vertex has exactly $K$ children. Then all vertices except the root  of  a $K$-ary tree have degree $K+1$, and the root has degree $K$. In this paper we say that a tree is {\it regular} if it is a $K$-ary tree for some $K\geq 1$.

For $x\in X$, let $|x|$ be the distance from the root $0$ to $x$, that is, the length of the geodesic from $0$ to $x$, where the length of every edge is $1$ and we consider each edge to be an isometric copy of the unit interval. The geodesic connecting two points $x, y\in V$ is denoted by $[x, y]$, and its length is denoted $|x-y|$. If $|x|<|y|$ and $x$ lies on the geodesic connecting $0$ to $y$, we write $x<y$ and call $y$ a descendant of the point $x$. More generally, we write $x\leq y$ if the geodesic from $0$ to $y$ passes through $x$, and in this case $|x-y|=|y|-|x|$.

On our $K$-regular tree $X$, we define the metric $ds$ and measure $d\mu$ by setting
\begin{equation*}
d\mu=\mu(|x|)\,d|x|,\ \  ds(x)=\lambda(|x|)\, d|x|,
\end{equation*}
where $\lambda, \mu:[0, \infty)\rightarrow (0, \infty)$ with $\lambda, \mu\in L^1_{\rm loc}([0, \infty))$.   
Here $d\,|x|$ is the measure which gives each edge Lebesgue measure $1$, as we consider each edge to be an isometric copy of the unit interval and the vertices are the end points of this interval. Hence for any two points $z, y\in X$, the distance between them is 
\[d(z, y)=\int_{[z, y]} \,ds(x)=\int_{[z, y]}\lambda(|x|)\, d|x|,\]
where $[z, y]$ is the unique geodesic from $z$ to $y$ in $X$.

We abuse the notation and let $\mu(x)$ and $\lambda(x)$ denote $\mu(|x|)$ and $\lambda(|x|)$, respectively, for any  $x\in X$, if there is no danger of confusion.



Throughout the paper, 
we let 
\[B(x,r)=\{y\in X: d(x,y)<r\}
\]denote the (open) ball in $X$ with center $x$ and radius $r$,  and let $\sigma B(x,r)=B(x, \sigma r)$. Also  
\[F(x,r)=\{y\in X: x\in[0,y], d(x,y)<r\}
\] is the downward directed half ball. For any $x\in X$ and $r>0$, we denote by $\ox^r$ the point in $[0, x]$ with $d(\ox^r, x)=\min\{r, d(0,x)\}$ 
and denote by $\ux_r$ a point in $X$ such that $x\in [0, \ux_r]$ with $d(\ux_r, x)=r$. Hence $\ox^r$ is the ancestor of any point $y\in B(x, r)$. Usually, the choice of $\ux_r$ is not unique, but we will not specify it since the results and proofs in this paper are independent of the choice of $\ux_r$.
\subsection{Admissibility}\label{adm}
{Let $u\in L_{\rm loc}^1(X)$. We say that a Borel function $g: X\rarrow [0, \infty]$ is an {\it upper gradient} of $u$ if  
\begin{equation}\label{gradient}|u(z)-u(y)|\leq \int_{\gamma} g\, ds
\end{equation}
whenever $z, y\in X$ and $\gamma$ is the geodesic from $z$ to $y$. In the setting of a tree any rectifiable curve with end points $z$ and $y$ contains the geodesic connecting $z$ and $y$, and therefore the upper gradient defined above is equivalent to the definition which requires that inequality \eqref{gradient} holds for all rectifiable curves with end points $z$ and $y$. 
In \cite{H03,HKST15}, the notion of a $p$-weak upper gradient is given. A Borel function $g: X\rightarrow [0, \infty]$ is called a $p$-weak upper gradient of $u$ if \eqref{gradient} holds on $p$-a.e. curve. Here we say that a property holds for {\it $p$-a.e. curve} if it fails only for a rectifiable curve family $\Gamma$ with {\it zero $p$-modulus}, i.e., there is Borel function $0\leq \rho\in L^p(X)$ such that $\int_{\gamma}\rho\, ds=\infty$ for every curve $\gamma\in \Gamma$.
We refer to \cite{H03,HKST15} for more information about $p$-weak upper gradients.

The notion of upper gradients is due to Heinonen and Koskela \cite{HK}; we refer interested readers to \cite{BB11,H03,HKST15,N00} for a more detailed discussion on upper gradients.  

The Newtonian space $N^{1,p}(X)$, for  $1\leq p<\infty$, is defined as the collection  of the functions for which the given norm
 \[\| u\|_{N^{1,p}(X)}:=\left (\int_{X}|u|^pd\mu+ \inf_{g}\int_{X}|g|^pd\mu  \right )^{1/p}
 \]
 is finite, where the infimum is taken over all $p$-weak upper gradients $g$ of $u$.

A measure $\mu$ is doubling if there exists a positive constant $C_d$ such that for all balls $B(x,r)$  with $x\in X$ and $r>0$, 
\begin{equation}\label{doubling}
\mu(B(x,2r))\leq C_d\mu(B(x,r)),
\end{equation}
where the constant $C_d$ is called the {\it doubling constant}.

$(X, d, \mu)$ supports a $(1, p)$-Poincar\'e inequality  if there exist  positive constants $C_P>0$ and $\sigma\geq 1$ such that for all balls $B(x,r)$ with $x\in X$ and $r>0$, every integrable function $u$ on $\sigma B(x,r)$ and all upper gradients $g$,
\begin{equation}\label{poincare}
\dashint_{B(x,r)}|u-u_{B(x,r)}|\,d\mu \leq C_Pr\left (\dashint_{\sigma B(x,r)}g^p\,d\mu \right )^{1/p}
\end{equation}
where  $u_B:=\dashint_{B}u\, d\mu=\frac{1}{\mu(B)}\int_B u\,d\mu$. We say that $\mu$ is {\it $p$-admissible} if $\mu$ is a doubling measure and $(X, d,\mu)$ supports a $(1, p)$-Poincar\'e inequality.}

The doubling property \eqref{doubling} and $(1, p)$-Poincar\'e inequality \eqref{poincare} can be defined on general metric measure spaces. In particular, on $\mathbb R^n$, in view of \cite[Theorem 2]{Keith03} or \cite[Theorem 8.4.2]{HKST15}, the $(1, p)$-Poincar\'e inequality \eqref{poincare} is equivalent to the $(1, p)$-Poincar\'e inequality given in the Introduction. It perhaps worth to point out that, since our $K$-regular trees are geodesic spaces, if $\mu$ is $p$-admissible, the dilation constant $\sigma$ in \eqref{poincare} can be taken to $1$, see \cite{PP1995} and \cite{HK00}.
\section{Proof of Theorem \ref{main theorem} for $K\geq 2$}\label{proof_2}
In this section, 
we give the proof of Theorem \ref{main theorem} for $K\geq 2$. To do so, we establish the following lemmas.
\begin{lem}\label{Lemma-tree-1}  Let $1\leq p<\infty$ and $X$ be a $K$-regular tree with distance $d$ and measure $\mu$ where $K\geq 1$. Assume that $\mu$ satisfies the \Ap-condition. Then $\mu$ is $p$-admissible.
\end{lem}

\begin{proof}
For $1\leq p<\infty$, let
	\[C_A:=\sup\left\{\text{A}_\text{p}(x, r): x\in X, r>0\right\}.\]
	Since $\mu$ satisfies the \Ap -condition, $0<C_A<\infty$. 

\underline{\bf Case $p=1$:}
	We first show that $\mu$ is a doubling measure. 
Let $x\in X$ and $r>0$ be arbitrary. Notice that $\text{A}_\text{1}(x, 2r)\leq C_A$. Then it follows from  \eqref{a1} that
\[\esssup_{w\in {[x,\, \underline{x}_{2r}]}}\frac{\lambda(w)}{K^{j(w)-j(x)}\,\mu(w)}\leq \frac{4rC_A}{\mu(F(\bar{x}^{2r},4r))}.\] 
Hence
\begin{align}\notag
r=\int_{[x,\underline{x}_r]}\, ds=& \int_{[x,\underline{x}_r]}\left (\frac{K^{j(w)-j(x)}\mu(w)}{\lambda(w)}\right )\left (\frac{\lambda(w)}{K^{j(w)-j(x)}\mu(w)}\right )ds(w)\\
&\leq \left (\int_{[x,\underline{x}_r]}\frac{K^{j(w)-j(x)}\mu(w)}{\lambda(w)}ds(w)\right ) \left (\frac{4rC_A}{\mu(F(\bar{x}^{2r},4r))}\right ).\label{0-lem41}
\end{align}
Notice that
\begin{equation*}
\int_{[x,\underline{x}_r]}\frac{K^{j(w)-j(x)}\mu(w)}{\lambda(w)}ds(w) =\mu(F(x,r))\leq \mu(B(x,r))
\end{equation*}
and that
\begin{equation*}
\mu(F(\bar{x}^{2r},4r))\geq \mu(B(x,2r)).
\end{equation*}
It follows from estimate \eqref{0-lem41} that 
\[r\leq 4C_Ar \frac{\mu(B(x,r))}{\mu(B(x,2r))},
\]
which proves that $\mu$ is a doubling measure with doubling constant $4C_A$ since $r>0$ and the pair $(x, r)$ is arbitrary.

Next we prove that $(X, d, \mu)$ supports a $(1, 1)$-Poincar\'e inequality. Consider an arbitrary ball $B(x,r)$ with $x\in X$ and $r>0$.
 By the triangle inequality, we obtain that
\begin{equation}\label{3-lem41}\dashint_{B(x,r)}|u-u_{B(x,r)}|d\mu \leq 2\dashint_{B(x,r)} |u(y)-u(\bar {x}^r)|d\mu(y)
\end{equation}
for the left-hand side of our Poincar\'e inequality.
By the definition of upper gradients and the Fubini theorem, for any upper gradient $g_u$ of $u$, the right-hand side of $(\ref{3-lem41})$ rewrites  as 
\begin{align}\notag 2\dashint_{B(x,r)}|u(y)-u(\bar {x}^r)|d\mu(y)&\leq 2\dashint_{B(x,r)}\int_{[\bar {x}^r,y]}g_u(w) ds(w)d\mu(y)\\
&=2\dashint_{B(x,r)}g_u(w) \frac{\lambda(w)}{\mu(w)} \left(\int_{B(x, r)} \chi_{[\ox^r, y]}(w)\, d\mu(y)\right)\, d\mu(w)\notag\\
\label{4-lem41} &=2\dashint_{B(x,r)}g_u(w) \frac{\lambda(w)}{\mu(w)}\mu(\{y\in B(x,r):w\in[0,y]\})d\mu(w).
\end{align}
Here the last equality holds since $\chi_{[\ox^r, y]}(w)$ is not zero only if $w\in [0,y]$.

Since the measure $\mu$ satisfies the $\text{A}_\text{1}$-condition, $\text{A}_\text{1}(\ox^r, 2r)<C_A$. It follows from \eqref{a1} that
\[
\frac{\mu(F(\ox^{3r}, 4r))}{4r}\cdot \esssup_{w\in [\ox^r, \ux_r]}\frac{\lambda(w)}{K^{j(w)-j(\ox^r)}\mu(w)}\leq C_A.
\]
Combining with the fact that $K^{j(\bar {x}^{3r})}\leq K^{j(\bar {x}^{r})}$, we obtain that
\begin{align}
\notag \frac{\lambda(w)}{\mu(w)}\mu(\{y\in B(x,r): w\in[0,y]\})&= \frac{\lambda(w)}{\mu(w)}\int_{\{y\in [w,\underline{w}_r]\cap B(x,r)\}} \frac{{K^{j(y)-j(w)}\mu(y)}}{\lambda(y)}ds(y)\\
&\leq\frac{\lambda(w) K^{j(\bar {x}^{3r})}}{\mu(w)K^{j(w)}}\int_{[\bar {x}^{3r}, \underline{x}_r]}\frac{K^{j(y)-j(\bar {x}^{3r})}\mu(y)}{\lambda(y)}ds(y) \notag \\
&\leq \frac{\lambda(w)}{\mu(w)K^{j(w)-j(\bar {x}^{r})}}\mu(F(\bar {x}^{3r},4r))\leq 4C_Ar \label{6-lem41}
\end{align} 
for any  $w\in B(x,r)$. Combining \eqref{3-lem41}-\eqref{6-lem41}, yields
\[\dashint_{B(x,r)}|u-u_{B(x,r)}|d\mu\leq 8C_Ar\dashint_{B(x,r)}g_u\, d\mu
\]for all balls $B(x,r)$.

\underline{\bf Case $p>1$:}
	Let us first prove that $\mu$ is a doubling measure. Let $B(x,r)$ be an arbitrary ball in $X$.
Since $\mu$ satisfies the \Ap-condition, we have $\text{\Ap}(x,2r)\leq C_A$, and hence
\begin{equation}
\label{eq1-lem41} \frac{\mu(F(\bar {x}^{2r},4r))}{4r} \cdot \left[\frac{1}{2r}\int_{[x,\underline{x}_{2r}]}\left (\frac{K^{j(w)-j(x)}\mu(w)}{\lambda(w)}\right)^{\frac{1}{1-p}}ds(w)\right]^{p-1}\leq C_A .
\end{equation} 
A simple calculation using the H\"older inequality shows that 
\begin{align}\notag
r=& \int_{[x,\underline{x}_r]}\left (\frac{K^{j(w)-j(x)}\mu(w)}{\lambda(w)}\right )^{1/p}\left (\frac{K^{j(w)-j(x)}\mu(w)}{\lambda(w)}\right )^{-1/p}ds(w)\\\notag
&\leq \left (\int_{[x,\underline{x}_r]}\frac{K^{j(w)-j(x)}\mu(w)}{\lambda(w)}ds(w)\right )^{1/p} \left [\int_{[x,\underline{x}_r]}\left (\frac{K^{j(w)-j(x)}\mu(w)}{\lambda(w)}\right )^{\frac{1}{1-p}}ds(w)\right ]^{\frac{p-1}{p}}\\
&\leq \mu(F(x,r))^{{1}/{p}} (2r)^{\frac{p-1}{p}}\left [\frac{1}{2r}\int_{[x,\underline{x}_{2r}]}\left (\frac{K^{j(w)-j(x)}\mu(w)}{\lambda(w)}\right )^{\frac{1}{1-p}}ds(w)\right ]^{\frac{p-1}{p}}.\notag
\end{align} 
Inserting  \eqref{eq1-lem41} into the above estimate yields 
\begin{equation}\label{label1}
r\leq (2r)^{\frac{p-1}{p}}\mu(F(x,r))^{1/p}\left [\frac{\mu(F(\bar{x}^{2r},4r))}{4rC_A}\right ]^{\frac{-1}{p}}={C_A}^{1/p}2^{\frac{p+1}{p}}r\left (\frac{\mu(F(x,r))}{\mu(F(\bar{x}^{2r},4r))}\right )^{1/p}.
\end{equation}
Note that $\mu(F(x,r))\leq \mu(B(x,r))$ and $\mu(F(\bar{x}^{2r},4r))\geq \mu(B(x,2r))$. Then the estimate \eqref{label1} implies that
\[r\leq {C_A}^{1/p}2^{\frac{p+1}{p}}r\left (\frac{\mu(B(x,r))}{\mu(B(x,2r))}\right )^{1/p},
\] 
which gives that $\mu$ is a doubling measure with doubling constant $C_A2^{p+1}$, since $r>0$ and $B(x, r)$ is arbitrary.

Next we show that $(X, d, \mu)$ supports a $(1, p)$-Poincar\'e inequality. Suppose $B(x,r)$ is an arbitrary ball with center $x\in X$ and radius $r>0$. 
 Since the measure $\mu$ satisfies the $\text{A}_\text{p}$-condition, then $\text{A}_\text{p}(\ox^r, 2r)<C_A$. It follows from \eqref{ap} that 
\begin{equation}
\frac{\mu(F(\ox^{3r},4r))}{4r}\cdot \left [\frac{1}{2r}\int_{[\ox^r,\ux_r]}\left(\frac{K^{j(w)-j(\ox^r)}\mu(w)}{\lambda(w)}\right)^{\frac{1}{1-p}}ds(w)\right ]^{{p-1}}\leq C_A\label{eq8-lem41}
\end{equation}

Recall that the left-hand side of our Poincar\'e inequality can be estimated by \eqref{4-lem41}. 
A simple calculation  shows that 
\begin{align}
\notag \frac{\lambda(w)}{\mu(w)}\mu(\{y\in B(x,r):w\in[0,y]\})&=\frac{\lambda(w)}{\mu(w)}\int_{\{y\in [w,\underline{w}_r]\cap B(x,r)\}}\frac{K^{j(y)-j(w)}\mu(y)}{\lambda(y)}ds(y)\\\notag
&\leq \frac{\lambda(w)}{\mu(w)K^{j(w)-j(\ox^{r})}}\int_{[\ox^{r},\ux_r]}\frac{K^{j(y)-j(\ox^{r})}\mu(y)}{\lambda(y)}ds(y)\\ \label{eq11-lem41}
&=\frac{\lambda(w)}{\mu(w)K^{j(w)-j(\ox^{r})}}\mu(F(\ox^{r},2r))
\end{align} 
for any point $w\in B(x,r)$. Inserting the estimate \eqref{eq11-lem41} into \eqref{4-lem41} yields that
\begin{align}
\notag \dashint_{B(x,r)}|u-u_{B(x,r)}|d\mu\leq 2\left ( \dashint_{B(x,r)}g_u(w)\frac{\lambda(w)}{\mu(w)K^{j(w)-j(\ox^{r})}}d\mu(w)\right )\mu(F(\ox^{r},2r)).
\end{align}Applying the H\"older inequality for the right-hand side of the above inequality, it follows that 
\begin{align}\notag
&\dashint_{B(x,r)}|u-u_{B(x,r)}|d\mu \\
\leq &2 \left (\dashint_{B(x,r)}{g_u}^p d\mu\right )^{1/p}\left [\dashint_{B(x,r)}\left (\frac{\lambda(w)}{K^{j(w)-j(\ox^{r})}\mu(w)}\right )^{\frac{p}{p-1}}d\mu(w)\right ]^{\frac{p-1}{p}}\mu(F(\ox^{r},2r)).\label{eq12-lem41}
\end{align}
By using the estimate \eqref{eq8-lem41}, we obtain that
\begin{align}
\notag & \left [\dashint_{B(x,r)} \left (\frac{\lambda(w)}{K^{j(w)-j(\ox^{r})}\mu(w)}\right )^{\frac{p}{p-1}}d\mu(w)\right ]^{\frac{p-1}{p}}\mu(F(\ox^{r},2r))\\
\notag 
\leq& \frac{\mu(F(\ox^{r},2r))}{\mu(B(x,r))^{\frac{p-1}{p}}}\left [\int_{F(\ox^r, 2r)}\left (\frac{\lambda(w)}{K^{j(w)-j(\ox^{r})}\mu(w)}\right )^{\frac{p}{p-1}}d\mu(w)\right ]^{\frac{p-1}{p}}\\
\notag \leq & \frac{\mu(F(\ox^{r},2r))}{\mu(B(x,r))^{\frac{p-1}{p}}}(2r)^{\frac{p-1}{p}}\left [\frac{1}{2r}\int_{[\ox^{r},\ux_r]}\left (\frac{K^{j(w)-j(\ox^{r})}\mu(w)}{\lambda(w)}\right )^{\frac{1}{1-p}}ds(w)\right ]^{\frac{p-1}{p}}\\
 \notag \leq & \frac{\mu(F(\ox^{r},2r))}{\mu(B(x,r))^{\frac{p-1}{p}}}(2r)^{\frac{p-1}{p}} \left [\frac{\mu(F(\ox^{3r},4r))}{4rC_A}\right ]^{\frac{-1}{p}}\\ 
 =&{C_A}^{1/p}2^{\frac{p+1}{p}}r\frac{\mu(F(\ox^{r},2r))}{\mu(B(x,r))^{\frac{p-1}{p}}\mu(F(\ox^{3r},4r))^{1/p}}.\label{eq13-lem41}
\end{align}
Note that $F(\ox^{r},2r))\subset B(x, 4r)$ and that $B(x, r)\subset F(\ox^{3r},4r)$.
Since $\mu$ is a doubling measure with doubling constant $C_A2^{p+1}$, we have that 
\[\frac{\mu(F(\ox^{r},2r))}{\mu(B(x,r))^{\frac{p-1}{p}}\mu(F(\ox^{3r},4r))^{1/p}}\leq \frac{\mu(B(x,4r))}{\mu(B(x,r))}\leq  (C_A 2^{p+1})^2.
\]
Inserting the above estimate into the estimate $(\ref{eq13-lem41})$, we have 
\begin{equation}
\label{eq14-lem41} \left [\dashint_{B(x,r)} \left (\frac{\lambda(w)}{K^{j(w)-j(\ox^{r})}\mu(w)}\right )^{\frac{p}{p-1}}d\mu(w)\right ]^{\frac{p-1}{p}}\mu(F(\ox^{r},2r))\leq {C_A}^{2+\frac{1}{p}}2^{\frac{p+1}{p}+2(p+1)}r.
\end{equation}
Thanks to the estimates $(\ref{eq12-lem41})$ and $(\ref{eq14-lem41})$, we obtain
\[\dashint_{B(x,r)}|u-u_{B(x,r)}|d\mu \leq {C_A}^{2+\frac{1}{p}}2^{\frac{1}{p}+2p+4}r\left (\dashint_{B(x,r)}{g_u}^p\, d\mu\right )^{\frac{1}{p}}
\]
for all balls $B(x,r)$.
\end{proof}
\begin{lem}\label{Lemma-tree-2}
	Let $1\leq p<\infty$ and $X$ be a $K$-regular tree with distance $d$ and measure $\mu$ where $K\geq 2$. Suppose that $\mu$ is $p$-admissible. Then $\mu$ satisfies the \Ap-condition.
\end{lem}

\begin{proof}
Let $x\in X$ and $r>0$ be arbitrary. Let $\varepsilon$ be an arbitrary positive number. Let $x_1\in X$ be a closest vertex of $x$ with $|x_1|>|x|$. Then we define 
\[T_{x_1}:=\{y\in X: x_1\in[0,y]\}\ \ \ \text{and}\ \ \ T_1:=[x, x_1]\cup T_{x_1}\]
Since $\mu$ is $p$-admissible, we may assume that $\mu$ satisfies the doubling condition \eqref{doubling} and the $(1, p)$-Poincar\'e inequality \eqref{poincare}.

\underline{\bf Case $p=1$:} 
 Let 
\[m=\essinf_{w\in [x, \ux_{\frac{r}{2}}]}\frac{K^{j(w)-j(x)}\mu(w)}{\lambda(w)}.
\]
In order to test the $(1,1)$-Poincar\'e inequality \eqref{poincare}, we define
\[u(y)=\begin{cases}0 & \text{ if }y\in X\setminus T_1,\\
\int_{[x,y]}{\chi_{E_{\varepsilon}}(w)}ds(w)& \text{ if }y\in F(x,r/2)\cap T_1,\\
a&\text{otherwise}	
\end{cases}
\]where $E_\varepsilon:=\left \{w\in F(x,{\frac{r}{2}}): \frac{K^{j(w)-j(x)}\mu(w)}{\lambda(w)}<m+\varepsilon\right \}$ and $a=\int_{[x,\ux_{\frac{r}{2}}]}\chi_{E_{\varepsilon}}(w)\,ds(w)$.
Note that $E_\varepsilon$ is a non-empty set by the definition of $m$ and that 
\begin{equation*}
r>a=\int_{[x,\ux_{\frac{r}{2}}]}\chi_{E_{\varepsilon}}(w)ds(w)>0.
\end{equation*}
By the definition of $u$, we obtain that  $g_u:=\chi_{E_{\varepsilon}}$ is an upper gradient of $u$.
Hence the right-hand side of the $(1,1)$-Poincar\'e inequality \eqref{poincare} is 
\begin{align*}
C_Pr\dashint_{\sigma B(x,r)}g_u\, d\mu&=C_P r\dashint_{\sigma B(x,r)}\chi_{E_\varepsilon}(w)d\mu(w)\\
&=\frac{C_P r}{\mu(\sigma B(x,r))}\int_{F(x,r/2)}\chi_{E_\varepsilon}(w)d\mu(w)\\
&=\frac{C_P r}{\mu(\sigma B(x,r))}\int_{[x,\ux_{\frac{r}{2}}]} \chi_{E_\varepsilon}(w)\frac{K^{j(w)-j(x)}\mu(w)}{\lambda(w)}\,ds(w).
\end{align*}
Here the second equality holds since $\chi_{E_\varepsilon}(w)$ is non-zero only if $w\in F(x, r/2)$.
Note that  $\mu(\sigma B(x,r))\geq \mu(B(x,r))$.  Then it follows from the definition of $E_\varepsilon$ that 
\begin{equation}
C_P r\dashint_{\sigma B(x,r)}g_u\, d\mu\leq \frac{C_P r}{\mu(B(x,r))}(m+\varepsilon)a.\label{lemma3.22}
\end{equation}
Let 
\begin{equation}\label{E1E2}
E_1:=B(x,r)\setminus T_1\ \ \text{and}\ \ E_2:=T_1\cap F(x,r)\setminus F(x,r/2).
\end{equation}
Note that $u\equiv 0$ on $E_1$ and $u\equiv a$ on $E_2$. 
Hence, at least one of the following holds:
\begin{equation}\label{lemma3.23}
|u-u_{B(x,r)}|\geq \frac{a}{2} \ \text{on}\ E_1\ \ \ \text{or}\ \ \ |u-u_{B(x,r)}|\geq \frac{a}{2} \ \text{on}\ E_2.
\end{equation}
Since $K\geq 2$, then $E_1$ and $E_2$ are not empty. Notice that  $K\mu(E_2)\geq \mu(F(x,r)\setminus F(x,r/2))$. 
Furthermore, the doubling property of $\mu$ gives
\[K\mu(E_2)\geq \mu(F(x,r)\setminus F(x,r/2))\geq \mu(B(\ux_{\frac{3r}{4}},r/4))\geq {C_d}^{-4}\mu(B(\ux_{\frac{3r}{4}},4r))\geq {C_d}^{-4}\mu(B(x,r))
\]
and
\[\mu(E_1)\geq \mu(B(z, r/2))\geq {C_d}^{-3}\mu(B(z, 4r))\geq {C_d}^{-3}\mu(B(x,r)),
\]
for some $z\notin T_1$ with $d(x, z)=r/2$.
Consequently, 
\begin{equation}\label{lemma3.24}
\min\{\mu(E_1), \mu(E_2)\}\geq {C_d}^{-4}K^{-1}\mu(B(x,r)).
\end{equation}
Then it follows from \eqref{lemma3.23} and \eqref{lemma3.24} that the left-hand side of the $(1,1)$-Poincar\'e inequality \eqref{poincare} is
\begin{align}\notag
\dashint_{B(x,r)}|u-u_{B(x,r)}|d\mu&\geq \frac{1}{\mu(B(x,r))}\max\left \{\int_{E_1}|u-u_{B(x,r)}|d\mu, \int_{E_2}|u-u_{B(x,r)}|d\mu\right \}\\
&\geq \frac{a}{2{C_d}^4 K}.\label{lemma3.25}
\end{align}
Combining the estimates \eqref{lemma3.22} and \eqref{lemma3.25}, we obtain that
\[\frac{a}{2{C_d}^4 K}\leq \frac{C_P r}{\mu(B(x,r))}(m+\varepsilon)a.
\]
Since $a>0$ and $\mu(F(\ox^{\frac{r}{2}},r))\leq \mu(B(x,2r))\leq C_d\mu(B(x,r))$, it follows that 
\[0<\frac{\mu(F(\ox^{\frac{r}{2}},r))}{r}\leq 2{C_d}^5 C_P K\cdot  (m+\varepsilon).
\]
 Since $\varepsilon$ and the pair $(x, r)$ are arbitrary, letting $\varepsilon\to 0$, the \Aone-condition holds.

\underline{\bf Case $p>1$:}
We define
 \[u(y)=\begin{cases}0 &\text{ if } y\in X\setminus T_1,\\
 \int_{[x,y]}\left (\frac{K^{j(w)-j(x)}\mu(w)}{\lambda(w)}\right )^{\frac{1}{1-p}}ds(w) & \text{ if }y\in F(x,r/2)\cap T_1,\\
 b& \text{ ortherwise}
 \end{cases}
 \] 
 where  
 \[b=\int_{[x,\ux_{\frac{r}{2}}]}\left (\frac{K^{j(w)-j(x)}\mu(w)}{\lambda(w)}\right )^{\frac{1}{1-p}}ds(w).
 \]
By the definition of $u$, we obtain that 
\begin{equation}\label{7-lem51}
g_u(y):=\left (\frac{K^{j(y)-j(x)}\mu(y)}{\lambda(y)}\right )^{\frac{1}{1-p}}\chi_{F(x,r/2)}(y)
\end{equation}
is an upper gradient of $u$.
Note that $u\equiv 0$ on $E_1$ and $u\equiv b$ on $E_2$ where $E_1$ and $E_2$ are defined as for $p=1$. Therefore, by an argument similar to the one in $p=1$ case, the left-hand side of the $(1, p)$-Poincar\'e inequality \eqref{poincare} can be estimated as
\begin{equation}
\label{lefhand}\dashint_{B(x,r)}|u-u_{B(x,r)}|d\mu \geq \frac{b}{2{C_d}^4 K}.
\end{equation}
For the right-hand side, we have that
\begin{align*}
C_Pr\left (\dashint_{\sigma B(x,r)}{g_u}^p\,d\mu\right )^{1/p}&=\frac{C_Pr}{\mu(\sigma B(x,r))^{1/p}}\left [\int_{F(x,r/2)}\left (\frac{K^{j(y)-j(x)}\mu(y)}{\lambda(y)}\right )^{\frac{p}{1-p}}d\mu(y)\right ]^{1/p}\\
&=\frac{C_Pr}{\mu(\sigma B(x,r))^{1/p}}\left [\int_{[x,\ux_{\frac{r}{2}}]}\left (\frac{K^{j(y)-j(x)}\mu(y)}{\lambda(y)}\right )^{\frac{1}{1-p}}ds(y)\right ]^{1/p}\\
&=\frac{C_Pr}{\mu(\sigma B(x,r))^{1/p}}b^{1/p}.
\end{align*}
Since $\mu(\sigma B(x,r))\geq \mu(B(x,r))$, it follows that
\begin{equation}
\label{righthand} C_Pr\left (\dashint_{\sigma B(x,r)}{g_u}^p\,d\mu\right )^{1/p}\leq \frac{C_Pr}{\mu(B(x,r))^{1/p}}b^{1/p}.
\end{equation}
Combining \eqref{lefhand} and $\eqref{righthand}$, we obtain that
\[\frac{b}{2{C_d}^4 K}\leq \frac{C_Pr}{\mu(B(x,r))^{1/p}}b^{1/p}.
\]
Notice that $\mu(F(\ox^{\frac{r}{2}},r))\leq \mu(B(x,2r))\leq C_d\mu(B(x,r))$. Hence we have
\[0<\frac{\mu(F(\ox^{\frac{r}{2}},r))^{1/p}}{r}\leq 2{C_d}^{4+\frac{1}{p}} C_p K b^{\frac{1-p}{p}}.
\]
Recalling the definition of $b$, the above estimate can be rewritten as
\[0< \frac{\mu(F(\ox^{\frac{r}{2}},r))}{r} \leq 2^p {C_d}^{4p+1} {C_P}^p K^p \left(\frac{1}{r}\int_{[x,\ux_{\frac{r}{2}}]}\left (\frac{K^{j(w)-j(x)}\mu(w)}{\lambda(w)}\right )^{\frac{1}{1-p}}ds(w)\right)^{1-p}.\]
Since the pair $(x, r)$ is arbitrary, the above estimate implies that $\mu$ satisfies the \Ap-condition. 
\end{proof}	
\begin{proof}[Proof of Theorem \ref{main theorem} for $K\geq 2$] The proof follows from Lemma \ref{Lemma-tree-1} and Lemma \ref{Lemma-tree-2}.
\end{proof}
\section{Proof of Theorem \ref{main theorem} for $K=1$}\label{section4}
\begin{lem}\label{lemma51}
	Let $1\leq p<\infty$ and $X$ be a $1$-regular tree with distance $d$ and measure $\mu$. 
	Suppose that $\mu$ is $p$-admissible. Then $\mu$ satisfies the $\text{A}_\text{p}$-condition far from $0$, i.e., 
		\[\sup\left \{\text{\Ap}(x,r): x\in X, 0<r\leq8d(0,x)\right \}<\infty.\]
\end{lem}
\begin{proof}
Let $(x,r)$ be an arbitrary pair with $d(0,x)\geq r/16>0$. Since $K=1$, we may let $T_1:=F(x, \infty)=\{y\in X: |y|\geq |x|\}$ and repeat the proof of Lemma \ref{Lemma-tree-2}. The only danger is whether \eqref{lemma3.24} holds, since, for $K=1$, $E_1$ could be empty. But here we required that $d(0,x)\geq r/16>0$, which gives a version of \eqref{lemma3.24}. Then the proof of Lemma \ref{Lemma-tree-2} gives that $\text{\Ap}(x, \frac r2)\leq C(p, K, C_d, C_P)$, where $C(p, K, C_d, C_P)$ is a constant only depending on $p, K, C_d$ and $C_P$. Since the pair $(x,r)$ is  arbitrary  with $d(0,x)\geq r/16>0$, we obtain that
\[\sup\left \{\text{\Ap}\left(x,\frac r2\right): x\in X, 0<\frac r2\leq8d(0,x)\right \}<\infty,\]
which gives the result.
\end{proof}

\begin{lem}\label{lemma4.3}
	Let $1\leq p<\infty$ and $X$ be a $1$-regular tree with distance $d$ and measure $\mu$. Assume that $\mu$ satisfies the \Ap-condition far from $0$. Then we have:  
	\begin{enumerate}
		\item The measure $\mu$ is doubling.
		\item There exists a  positive constant $C_p>0$ such that for all balls $B(x,r)$ with $x\in X$ and $0<r\leq \frac{4}{5} d(0, x)$, every integrable function $u$ on $B(x,r)$ and all upper gradients $g$ of $u$,
\begin{equation}\label{poincareaway}
\dashint_{B(x,r)}|u-u_{B(x,r)}|\,d\mu \leq C_p r\left (\dashint_{B(x,r)}g^p\,d\mu \right )^{1/p}.
\end{equation}
\end{enumerate}\label{lemma52}
\end{lem}

\begin{proof}
\underline{\bf Claim $1$:}
Recall the proof  of Lemma $\ref{Lemma-tree-1}$. It actually shows that for any pair $(x,r)$ with $\text{A}_\text{p}(x, 2r)\leq C_A$, we have 
\[\mu(B(x, 2r))\leq C(C_A) \mu(B(x, r)),\]
where $C(C_A)$ is a constant only depending on $C_A$. In this lemma, since $\mu$ only satisfies the $\text{A}_\text{p}$-condition far from $0$, i.e.,  
\[M_A:=\sup\left\{\text{A}_\text{p}(x, r): x\in X, 0<r\leq 8\, d(0, x)\right\}<\infty,\]
we obtain that there is a positive constant $C:=C(M_A)$ only depending on $M_A$ such that
	\begin{equation}\label{10-lem52}
	\mu(B(x,r))\leq C\mu(B(x,r/2))
	\end{equation}
	for all balls $B(x,r)$ with $d(0,x)\geq r/8>0$.

 To get that $\mu$ is a doubling measure, it is sufficient to show that $(\ref{10-lem52})$ holds for all balls $B(x,r)$ with $d(0,x)<r/8$. 
	Note that $d(0,\underline{0}_{\frac{r}{2}})=\frac{r}{2}\geq \max\{4r/8, 2r/8, r/8\}$. Applying $(\ref{10-lem52})$ for $B(\underline{0}_{\frac{r}{2}},4r)$, $B(\underline{0}_{\frac{r}{2}}, 2r)$ and $B(\underline{0}_{\frac{r}{2}},r)$ in turns, we obtain that
	\begin{equation*}
	\mu(B(\underline{0}_{\frac{r}{2}},4r))\leq C\mu(B(\underline{0}_{\frac{r}{2}},2r))\leq C^2\mu(B(\underline{0}_{\frac{r}{2}},r))\leq C^3\mu(B(\underline{0}_{\frac{r}{2}},r/2)).
	\end{equation*}
Hence 
	\begin{equation}
	\mu(B(\underline{0}_{\frac{r}{2}},4r))\leq C^3\mu(B(\underline{0}_{\frac{r}{2}}, r/2)). \label{11-lem52}
	\end{equation}
	 From $B(\underline{0}_{\frac{r}{2}},r/2)\subset B(0,r)$ and $B(0,2r)\subset B(\underline{0}_{\frac{r}{2}},4r)$, we have 
	\[\mu(B(0,2r))\leq \mu(B(\underline{0}_{\frac{r}{2}},4r)), \ \ \mu(B(\underline{0}_{\frac{r}{2}},r/2))\leq \mu(B(0,r))
	\]for all $r>0$. Combining with $(\ref{11-lem52})$, we get that 
	\begin{equation*}
	\mu(B(0,2r))\leq C^3\mu(B(0,r))
	\end{equation*}for all $r>0$. In particular,
	\begin{equation}
	\label{12-lem52}\mu(B(0,2r))\leq C^9\mu(B(0,r/4))
	\end{equation} for all $r>0$. Let $B(x,r)$ be an arbitrary ball with $d(0,x)<r/8$. By $B(x,r)\subset B(0,2r)$ and $B(0,r/4)\subset B(x, r/2)$, it follows from $(\ref{12-lem52})$  that 
	\begin{equation}
\notag \mu(B(x,r))\leq \mu(B(0,2r))\leq C^9 \mu(B(0, r/4))\leq C^9\mu(B(x,r/2))
	\end{equation} for all balls $B(x,r)$ with $d(0,x)<r/8$. Combining with $(\ref{10-lem52})$, we conclude that $\mu$ is a doubling measure.

\underline{\bf Claim $2$:} Recall the proof  of Lemma $\ref{Lemma-tree-1}$. It actually shows that for any pair $(x,r)$ with $\text{A}_\text{p}(\ox^r, 2r)\leq C_A$, there exists a constant $C_p(C_A)$ such that for every integrable function $u$ on $B(x, r)$ and all upper gradients $g$ of $u$,  the $(1, p)$-Poincar\'e inequality \eqref{poincareaway} holds for $B(x, r)$,
where $C_p(C_A)$ is a constant only depending on $C_A$. In this lemma, $\mu$ only satisfies the $\text{A}_\text{p}$-condition far from $0$, i.e.,  
\[M_A:=\sup\left\{\text{A}_\text{p}(x, r): x\in X, 0<r\leq 8\, d(0, x)\right\}<\infty.\]
Since 
\[0<2r\leq 8\, d(0, \ox^r)\ \iff\ d(0, \ox^r)\geq r/4>0\ \iff\  d(0, x)\geq 5r/4>0,\]
we obtain that there is a positive constant $C_p:=C(M_A)$ only depending on $M_A$ such that the Claim $2$ holds.
\end{proof}

We say $(X, d, \mu)$ supports a {\it $(1, p)$-Poincar\'e inequality at $0$}, $1\leq p<\infty$, if there are  positive constants $C_0, \sigma_0\geq 1$ such that for any $r>0$, every integrable function $u$ on $\sigma_0 B(0, r)$ and all upper gradients $g$ of $u$,
\begin{equation}\label{poincare0}
\dashint_{B(0,r)}|u-u_{B(0,r)}|\,d\mu \leq C_0 r\left (\dashint_{\sigma_0 B(0,r)}g^p\,d\mu \right )^{1/p}.
\end{equation}

\begin{prop}\label{prop}
Let $1\leq p<\infty$ and $(X, d, \mu)$ be as in Lemma \ref{lemma4.3}. Assume additionally that $(X, d, \mu)$ supports a {$(1, p)$-Poincar\'e inequality at $0$}. Then $\mu$ is $p$-admissible.
\end{prop}
\begin{proof}	
	It follows from Claim $2$ of Lemma \ref{lemma4.3} that it suffices to check the $(1, p)$-Poincar\'e inequality \eqref{poincare} for balls $B(x, r)$ with $d(0, x)<5r/4$.

	Fix an arbitrary ball $B(x, r)$ with $d(0, x)<5r/4$. 
	By the triangle inequality, the left-hand side of a $(1, p)$-Poincar\'e inequality \eqref{poincare} can be estimated as
	\begin{equation}\label{triangle}
	\dashint_{B(x,r)}|u-u_{B(x,r)}|d\mu \leq 2\dashint_{B(x,r)}|u-u_{B(0,4r)}|d\mu. 
	\end{equation}
	 It follows from Claim $1$ of  Lemma \ref{lemma4.3} that $\mu$ is a doubling measure. Without loss of generality, we may assume that the doubling constant is $C_d$. Since $d(0, x)<5r/4$, then $B(0, 4r)\subset B(x, 8r)$. Hence by doubling property,
	 \[\mu(B(x,r))\geq {C_d}^{-3}\mu(B(x,8r))\geq {C_d}^{-3}\mu(B(0,4r)).\]
Combining with \eqref{poincare0}, the estimate \eqref{triangle} can be rewritten as
\begin{align}
\dashint_{B(x,r)}|u-u_{B(x,r)}|d\mu &\leq 2{C_d}^3 \dashint_{B(0, 4r)}|u-u_{B(0,4r)}|d\mu\notag\\
&\leq 8{C_d}^3 C_0 r \left(\dashint_{\sigma_0 B(0, 4r)}g^p\, d\mu\right)^{1/p}.\label{ffinal}
\end{align}
An easy verification shows that
\begin{equation}\label{final1}
\dashint_{\sigma_0 B(0, 4r)}g^p\, d\mu \leq {C_d}^2 \dashint_{\sigma_0 B(x, 8r)} g^p\, d\mu,
\end{equation}
since $\sigma_0 B(0, 4r)\subset \sigma_0 B(x, 8r)$ and $\mu(\sigma_0 B(x, 8r))\leq {C_d}^2 \mu(\sigma_0 B(x, 2r))\leq {C_d}^2 \mu(\sigma_0 B(0, 4r))$ by doubling.
Combining \eqref{ffinal} and \eqref{final1}, we deduce that
\begin{equation}\label{final}
\dashint_{B(x,r)}|u-u_{B(x,r)}|d\mu \leq 8{C_d}^{3+2/p} C_0 r \left(\dashint_ {8\sigma_0B(x,  r)} g^p\, d\mu\right)^{1/p}.
\end{equation}

Since $B(x, r)$ is an arbitrary ball with $d(0, x)<5r/4$, combining \eqref{final} with Claim $1$ and $2$ of Lemma \ref{lemma4.3}, it shows that $\mu$ is $p$-admissible.
\end{proof}

The following lemma shows that the assumption in Lemma \ref{lemma4.3} is sufficient to obtain a $(1, p)$-Poincar\'e inequality at $0$, which means that the additional assumption in Proposition \ref{prop} is redundant. The core idea of the proof comes from the proof of \cite[Theorem 1]{PP1995}.
\begin{lem}\label{lemma53}
	Let $1\leq p<\infty$ and $(X, d, \mu)$  be as in Lemma \ref{lemma4.3}. Then $(X, d,\mu)$ supports a $(1, p)$-Poincar\'e inequality at $0$.
\end{lem}
\begin{proof}
 It follows from Lemma \ref{lemma4.3} that $\mu$ is doubling and $(X, d, \mu)$ supports the $(1, p)$-Poincar\'e inequality \eqref{poincareaway}.
 For any $R>0$, since $X$ is a $1$-regular tree, we have  $B(0, R)=[0, x_R)$, where $x_R\in X$ with $|x_R|=R$. By using the geometry of the $1$-regular tree, we are able to modify the proof of \cite[Theorem 1]{PP1995} by using a better chain condition $\{B(x_i, r_i)\}_{i\in\mathbb N}$ which requires additionally that $r_i<\frac{4}{5}d(x_i, 0)$ (since \eqref{poincareaway} only works  for balls $B(x, r)$ with $r<\frac45d(x, 0)$). Hence it follows from the proof of \cite[Theorem 1]{PP1995} that  there is a constant $C$ independent of $R$ such that
 \[\dashint_{B(0, R)}|u-u_{B(0, R)}|\, d\mu\leq C R\left(\dashint_{B(0, R)} g^p\, d\mu\right)\]
 for all integrable functions $u$ and all upper gradients $g$ of $u$. 
\end{proof}

 \begin{proof}[Proof of Theorem \ref{main theorem} for $K=1$] The claim follows from Lemma \ref{lemma4.3}, Proposition \ref{prop} and Lemma \ref{lemma53}.
 \end{proof}

 \begin{rem}\label{remark}\rm
 Fix any $\infty>c>1$, if we change the \Ap-condition far from $0$, i.e., the condition \eqref{Ap2} to 
 \begin{equation}\label{newap}
 \sup\left \{\text{\Ap}(x,r): x\in X, 0<r\leq c\,d(0,x)\right \}<\infty,
 \end{equation}
repeating the proof Theorem \ref{main theorem} and related lemmas, it follows that the condition \eqref{newap} is also equivalent to $\mu$ being $p$-admissible.
 \end{rem}
	
\begin{example}\label{example}\rm
The following example from   \cite[Example 4]{AH19} or \cite[Example 6.2]{BBL} gives a $1$-regular tree with a non-doubling measure which satisfies \eqref{newap} for any $0<c<1$. Let $X=(\mathbb R_+, dx, \mu\, d\mu)$ with $\mu(x)=\min\{1, x^{-1}\}$. Then it follows from \cite{BBL} and \cite{AH19} that $\mu$ is not a doubling measure, hence $\mu$ is not $p$-admissible for any $1\leq p<\infty$. It remains to show that \eqref{newap} holds for any $0<c<1$ and $1\leq p<\infty$.

Fix $0<c<1$. Let $R= \frac{1}{1-c}$. To show \eqref{newap} holds, it suffices to show that 
\begin{equation}\label{veryn}
\sup \left \{\text{\Ap}(t,\beta t):  0<\beta\leq c, t\in (R, \infty)\right \}<\infty,
\end{equation}
since
\[\sup \left \{\text{\Ap}(t,\beta t):  0<\beta\leq c, t\in [0, R]\right \}<\infty\]
is given by the fact that $(R+cR)^{-1}\leq \mu(x)\leq 1$ for any $x\in F({\bar t}^{\beta t}, 2\beta t)$ with $t\leq R$ and $0<\beta\leq c$.
 For any $0<\beta\leq c$, since $F({\bar t}^{\beta t}, 2\beta t)=[t-\beta t, t+\beta t]$ and $t-\beta t>1$ for any $t>R$, we have that
\[\mu(F({\bar t}^{\beta t}, 2\beta t)) \leq \int_{(1-\beta)t}^{(1+\beta)t} x^{-1} \, dx= \log \left(\frac{1+\beta}{1-\beta}\right)\leq \log\left(\frac{1+c}{1-c}\right).\]
On the other hand, we have that for $p>1$, 
\[\left(\frac{1}{\beta t}\int_{t}^{t+\beta t} x^{\frac{1}{p-1}}\, dx\right)^{p-1} =t\left(\frac{(1+\beta)^{p/(p-1)}-1}{\beta}\right)^{p-1}\leq C(c, p) t,\]
where $C(c, p)$ is a constant only depending on $c$ and $p$, and that
\[\esssup_{x\in[t, t+\beta t]} x= (1+\beta)t\leq (1+c)t.\]
Hence condition \eqref{veryn} holds.
\end{example}
\section*{Acknowledgement}
The authors thank our advisor Professor Pekka Koskela for helpful discussions.

\noindent
Department of Mathematics and Statistics, University of Jyv\"askyl\"a, PO~Box~35, FI-40014 Jyv\"askyl\"a, Finland.

\medskip

\noindent Khanh Ngoc Nguyen

\noindent{\it E-mail address}:  \texttt{khanh.n.nguyen@jyu.fi}, \texttt{khanh.mimhus@gmail.com}

\noindent Zhuang Wang

\noindent{\it E-mail address}:  \texttt{zhuang.z.wang@jyu.fi}

\end{document}